\documentclass[a4paper,reqno]{amsart}

\usepackage[french,british]{babel}
\usepackage{amsthm}
\usepackage{amsmath}
\usepackage{amssymb}
\usepackage{mathrsfs}
\usepackage{centernot}   
\usepackage{graphicx}        
\usepackage{multicol}        
\usepackage[all,ps]{xy}
\usepackage{layout}
\usepackage{booktabs}

\newcommand{\cE}{{\mathscr E}}
\newcommand{\cF}{{\mathscr F}}

\newcommand{\cN}{{\mathscr N}}

\newcommand{\cO}{{\mathscr O}}
\newcommand{\cP}{{\mathscr P}}  

\newcommand{\cR}{{\mathscr R}}
\newcommand{\cS}{{\mathscr S}}
\newcommand{\cT}{{\mathscr T}}

\newcommand{\cZ}{{\mathscr Z}}

\newcommand{\es}{\emptyset}

\newcommand{\gz}{\mathfrak{z}}

\newcommand{\hra}{\hookrightarrow}
\newcommand{\KK}{{\mathbb K}}
\newcommand{\la}{\langle}

\newcommand{\lra}{\longrightarrow}

\newcommand{\ov}{\overline}
\newcommand{\PP}{{\mathbb P}}
\newcommand{\QQ}{{\mathbb Q}}
\newcommand{\ra}{\rangle}

\newcommand{\ZZ}{{\mathbb Z}}

\theoremstyle{plain}

\newtheorem{thm}{Theorem}[section]

\newtheorem{clm}[thm]{Claim}

\newtheorem{cnj}[thm]{Conjecture}
\newtheorem{crl}[thm]{Corollary}

\newtheorem{prp}[thm]{Proposition}
\newtheorem{prp-dfn}[thm]{Proposition-Definition}

\theoremstyle{definition}

\newtheorem{dfn}[thm]{Definition}

\theoremstyle{remark}

\newtheorem{expl}[thm]{Example}

\newtheorem{rmk}[thm]{Remark}

\DeclareMathOperator{\alb}{alb}
\DeclareMathOperator{\Alb}{Alb}

\DeclareMathOperator{\Aut}{Aut}

\DeclareMathOperator{\CH}{CH}

\DeclareMathOperator{\cod}{cod}

\DeclareMathOperator{\im}{Im}

\DeclareMathOperator{\Pic}{Pic}

\DeclareMathOperator{\Top}{Top}

\usepackage{ifthen}
\usepackage{rotating}
\usepackage{supertabular}
\newcommand{\cit}[1]{{\rm \textbf{#1}}}
\newcommand{\Ref}[2]{\cit{%
\ifthenelse{\equal{#1}{thm}}{Theorem}{}%
\ifthenelse{\equal{#1}{ass}}{Assumption}{}%
\ifthenelse{\equal{#1}{chp}}{Chapter}{}%
\ifthenelse{\equal{#1}{prp}}{Proposition}{}%
\ifthenelse{\equal{#1}{lmm}}{Lemma}{}%
\ifthenelse{\equal{#1}{cnj}}{Conjecture}{}%
\ifthenelse{\equal{#1}{crl}}{Corollary}{}%
\ifthenelse{\equal{#1}{dfn}}{Definition}{}%
\ifthenelse{\equal{#1}{expl}}{Example}{}%
\ifthenelse{\equal{#1}{hyp}}{Hypothesis}{}%
\ifthenelse{\equal{#1}{rmk}}{Remark}{}%
\ifthenelse{\equal{#1}{clm}}{Claim}{}%
\ifthenelse{\equal{#1}{exe}}{Exercise}{}%
\ifthenelse{\equal{#1}{qst}}{Question}{}%
\ifthenelse{\equal{#1}{sec}}{Section}{}%
\ifthenelse{\equal{#1}{subsec}}{Subsection}{}%
\ifthenelse{\equal{#1}{subsubsec}}{Subsubsection}{}%
\ifthenelse{\equal{#1}{univ}}{Universal Property}{}%
\ifthenelse{\equal{#1}{trm}}{Terminology}{}%
\ifthenelse{\equal{#1}{tbl}}{Table}{}%
\  \ref{#1:#2}%
}}

\usepackage{multirow}

 \makeindex
 
\begin{document}
 \title{Computations with modified diagonals}
 \author{Kieran G. O'Grady\\\\
\lq\lq Sapienza\rq\rq Universit\`a di Roma}
\dedicatory{Alla piccola Titti}
\date{March 23 2014}
\thanks{Supported by PRIN 2010}
\begin{abstract}
Motivated by conjectures of Beauville and Voisin on the Chow ring of Hyperk\"ahler varieties we will prove some basic results on the rational equivalence class of modified diagonals of projective varieties.


\end{abstract}
\maketitle
\tableofcontents
\section{Introduction}
Let $X$ be an $n$-dimensional  variety over a field $\KK$ and $a\in X(\KK)$.  For $I\subset\{1,\ldots,m\}$ we let
\begin{equation}\label{diagtwist}
\Delta^m_I(X;a):=\{(x_1,\ldots,x_m)\in X^m \mid \text{$x_i=x_j$ if $i,j\in I$ and $x_i=a$ if $i\notin I$}\}.
\end{equation}
 The \emph{$m$-th modified diagonal cycle associated to $a$} is the  $n$-cycle  on $X^m$ given by
\begin{equation}\label{eccogamma}
\Gamma^m(X;a):=\sum\limits_{\es\not= I\subset \{1,2,\ldots,m\}}(-1)^{m-|I|}\Delta^m_I(X;a)
\end{equation}
if $n>0$, and equal to $0$ if $n=0$. Gross and Schoen~\cite{groscho} proved that if $X$ is a  (smooth projective) hyperelliptic curve   and $a$ is a fixed point of a hyperelliptic involution then 
 $\Gamma^3(X;a)$ represents a torsion class in the Chow group of $X^3$.  On the other hand it is known that if $X$ is a generic complex smooth plane  curve and   $m$ is small compared to its genus then $\Gamma^m(X;a)$ is \emph{not} algebraically equivalent to $0$, whatever $a$ is, 
 see~\cite{voisinf} (for the link between vanishing of  $\Gamma^m(X;a)$ and Voisin's result on the Beauville decomposition of the Abel-Jacobi image of a curve see the proof of Prop.4.3 of~\cite{beauvoisin}).
Let $X$ be a complex projective $K3$ surface: Beauville and Voisin~\cite{beauvoisin} have proved that   there exists $c\in X$ such that the   
 rational equivalence class of $\Gamma^3(X;c)$ is torsion.
  A natural question arises: under which hypotheses a modified diagonal cycle on a projective variety represents a torsion class  in the Chow group?  We should point out that such a vanishing can entail unexpected geometric properties:  if $X$ is a smooth projective variety of dimension $n$ and  $\Gamma^{n+1}(X;a)$ is torsion in the Chow group then the intersection of arbitrary divisor classes $D_1,\ldots,D_{n}$ on $X$ is rationally equivalent to a multiple of $a$. A set of conjectures put forth by Beauville~\cite{beauconj} and Voisin~\cite{voisinhk} predict exactly such  a degenerate behaviour for the intersection product of divisors on hyperk\"ahler varieties i.e.~complex smooth projective varieties which are simply connected and carry a holomorphic symplectic form whose cohomology class  spans $H^{2,0}$ (see~\cite{liefu,shenvial} for more results on those conjectures). Our interest in modified diagonals has been motivated by the desire to prove the  conjecture on hyperk\"ahler varieties stated below. From now on the notation   $A\equiv B$  for cycles $A,B$ on a variety $X$   means that for some integer $d\not=0$ the cycle  $dA$ is rationally equivalent to $dB$, i.e.~we will work with the rational Chow group $\CH(X)_{\QQ}:=\CH(X)\otimes_{\ZZ}\QQ$.
\begin{cnj}\label{cnj:diaghk}
Let $X$ be a Hyperk\"ahler variety of dimension $2n$. Then there exists $a\in X$ such that $\Gamma^{2n+1}(X;a)\equiv 0$. 
\end{cnj}
In the present paper we will \emph{not} prove~\Ref{cnj}{diaghk}, instead we will establish a few basic results on modified diagonals. 
Below is our first result, see~\Ref{sec}{prodotti}.
\begin{prp}\label{prp:protokunn}
Let $X,Y$ be smooth projective varieties. Suppose that  there exist $a\in X(\KK)$, $b\in Y(\KK)$ such that  $\Gamma^m(X;a)\equiv 0$  and  $\Gamma^n(Y;b)\equiv 0$. Then  $\Gamma^{m+n-1}(X\times Y;(a,b))\equiv 0$.
\end{prp}
We will  apply the above proposition in order to show that if $T$ is a complex  abelian surface and $a\in T$ then  $\Gamma^5(T;a)\equiv 0$. Notice that if $E$ is an elliptic curve and $a\in E$ then     $\Gamma^3(E;a)\equiv 0$ by Gross and Schoen~\cite{groscho}. These results are particular instances of a Theorem of Moonen and Yin~\cite{moy} which asserts that  $\Gamma^{2g+1}(A;p)\equiv 0$ for $A$  an abelian variety  of dimension $g$ and $p\in A(\KK)$ (and more generally for an abelian scheme of relative dimension $g$).
A word about the relation between   Moonen - Yin's result and~\Ref{cnj}{diaghk}. Beauville and Voisin  proved that the relation $\Gamma^3(X;c)\equiv 0$  for $X$ a complex projective  $K3$ surface  (and a certain $c\in X$)  follows from the existence of an elliptic surface $Y$ dominating $X$ and the relation $\Gamma^3(E_t;a)\equiv 0$ for the fibers of the elliptic fibration on $Y$.   We expect that the theorem of Moonen and Yin can be used to prove 
that~\Ref{cnj}{diaghk} holds for Hyperk\"ahler varieties which are covered generically by abelian varieties, this is the subject of work in progress. (It is hard to believe that every Hyperk\"ahler variety of dimension greater than $2$ is covered generically by abelian varieties, but certainly there are  interesting codimension-$1$ families which have this property, viz.~lagrangian fibrations and Hilbert schemes of $K3$ surfaces, moreover Lang's conjectures on hyperbolicity would give that a hyperk\"ahler variety is generically covered by varieties birational to abelian varieties.) In~\Ref{sec}{fibrazioni} we will prove that, in a certain sense, \Ref{prp}{protokunn} holds also for $\PP^r$ fibrations over smooth projective varieties if  certain hypotheses are satisfied, then we will apply the result to prove vanishing of classes of modified diagonals of symmetric products of curves of genus at most $2$. In~\Ref{sec}{scoppio} we will prove 
 the following result.
\begin{prp}\label{prp:blowdel}
Let $Y$ be a smooth projective variety  and $V\subset Y$ be  a smooth subvariety of codimension $e$. Suppose that there exists $b\in V(\KK)$ such that $\Gamma^{n+1}(Y;b)\equiv 0$ and $\Gamma^{n-e+1}(V;b)\equiv 0$. Let $X\to Y$ be the blow-up of $V$ and $a\in X(\KK)$ such that $f(a)=b$. Then $\Gamma^{n+1}(X;a)\equiv 0$.
\end{prp}   
We will apply~\Ref{prp}{blowdel} and~\Ref{prp}{protokunn} in order to show that~\Ref{cnj}{diaghk} holds for $S^{[n]}$ where $S$ is a complex $K3$ surface and $n=2,3$, see~\Ref{prp}{diaghilbk3}. In~\Ref{sec}{rivdop} we will consider double covers $f\colon X\to Y$ where $X$ is a projective variety. We will prove that if $a\in X(\KK)$ is a  ramification point and $\Gamma^m(Y;f(a))\equiv 0$  then $\Gamma^{2m-1}(X;a)\equiv 0$, provided $m=2,3$. The proof for $m=2$ is  the proof, given by Gross and Schoen,  that if $X$ is a   hyperelliptic curve    then 
 $\Gamma^3(X;a)\equiv 0$ for  $a\in X(\KK)$  a fixed point of a hyperelliptic involution; we expect that our extension will work for arbitrary $m$ but we have not been able to carry out the necessary linear algebra computations. The  result for $m=3$ allows us to give another proof that
   $\Gamma^5(T;a)\equiv 0$ for a complex abelian surface $T$: the equality $\Gamma^5(T;a)\equiv 0$ follows from our result on double covers and  the equality  $\Gamma^3(T/\la -1\ra;c)\equiv 0$ proved by  Beauville and Voisin~\cite{beauvoisin}. 
\subsection{Conventions and notation}
Varieties are defined over a base field $\mathbb K$. A \emph{point of $X$} is an element of $X(\KK)$.   
We denote the small diagonal $\Delta^m_{\{1,\ldots,m\}}(X;a)$  by $\Delta^m(X)$ and we let $\pi^m_i\colon X^m\to X$ be the $i$-th projection - we will drop the superscript $m$ if there is no potential for confusion. We let $X^{(n)}$ be the $n$-th symmetric product of $X$ i.e.~$X^{(n)}:=X^n/{\mathcal S}_n$ where ${\mathcal S}_n$ is the symmetric group on $n$ elements.
\subsection{Acknowledgments}
It is a pleasure to thank Lie Fu, Ben Moonen and Charles Vial for the interest they took in this work.
\section{Preliminaries}
\subsection{}\label{subsec:significato}
\setcounter{equation}{0}
Let $X$ be an $n$-dimensional  projective variety over a field $\KK$, $a\in X(\KK)$ and   $h$  a hyperplane class on $X$. Let $\iota\colon\Delta^m(X)\hra X^m$ be the inclusion map. If $m\le n$ then
\begin{equation}\label{iperpiano}
\Gamma^{m}(X;a)\cdot\pi_1^{*}(h)\cdot\pi_2^{*}(h)\cdot\ldots\cdot\pi_{m-1}^{*}(h)\cdot\pi_{m}(h^{n-m+1})=
\iota_{*}(h^n).
\end{equation}
Since $\deg\iota_{*}(h^n)\not=0$ it follows that $\Gamma^{m}(X;a)\not\equiv 0$ if $m\le n$. Now suppose that $\Gamma^{n+1}(X;a)\equiv 0$. Let $D_1,\ldots,D_n$  
 be \emph{Cartier} divisors on $X$: then
\begin{equation}\label{tuttimulti}
0=\pi_{n+1,*}(\Gamma^{n+1}(X;a)\cdot\pi_1^{*}D_1\cdot\ldots\cdot \pi_n^{*}D_n)=D_1\cdot\ldots\cdot D_n-\deg(D_1\cdot\ldots\cdot D_n)a
\end{equation}
in $\CH_0(X)_{\QQ}$. 
\begin{rmk}\label{rmk:unicopunto}
Equation~\eqref{tuttimulti}  shows that if $\Gamma^{n+1}(X;a)\equiv 0$ and $\Gamma^{n+1}(X;b)\equiv 0$ then $a\equiv b$. 
\end{rmk}
\begin{expl}\label{expl:spapro}
The intersection product between cycle classes of complementary dimension defines a perfect pairing on $\CH((\PP^n)^m)$. Let $a\in\PP^n$: since $\Gamma^{n+1}(\PP^n;a)$ pairs to $0$ with any class of complementary dimension it follows that $\Gamma^{n+1}(\PP^n;a)\equiv 0$.  
\end{expl}
\subsection{}\label{subsec:homcomp}
\setcounter{equation}{0}
In the present subsection we will assume that $X$ is a complex smooth projective variety of dimension $n$. Let $a\in X$. Let $\alpha_1,\ldots,\alpha_m\in H_{DR}(X)$ be De Rham homogeneous cohomology classes such that $\sum_{i=1}^m\deg\alpha_i=2n$. Thus it makes sense to integrate $\pi_1^{*}\alpha_1\wedge\ldots\wedge\pi_m^{*}\alpha_m$ on $\Gamma^m(X;a)$.  Let
\begin{equation}\label{eccoesse}
s:=|\{1\le i\le m \mid \deg\alpha_i=0\}|.
\end{equation}
A straightforward computation gives that
\begin{equation}\label{integrale}
\int_{\Gamma^m(X;a)} \pi_1^{*}\alpha_1\wedge\ldots\wedge\pi_m^{*}\alpha_m =
 \sum_{\ell=0}^{m-1}(-1)^{\ell}{{s}\choose{\ell}}\int_X \alpha_1\wedge\ldots\wedge\alpha_m.
\end{equation}
\begin{prp}\label{prp:banale}
Let $X$ be a smooth complex projective variety  and $a\in X$. Let $n$ be the  dimension  of $X$ and $d$ be its Albanese dimension. The homology class of $\Gamma^{m}(X;a)$ is torsion if and only if $m>(n+d)$. 
\end{prp}
\begin{proof}
If $n=0$ the result is obvious. From now on we assume that $n>0$. By~\eqref{iperpiano} we may assume that $m>n$. 
The homology class of $\Gamma^{m}(X;a)$ is torsion if and only if the left-hand side of~\eqref{integrale} vanishes for every choice of homogeneous $\alpha_1,\ldots,\alpha_m\in H_{DR}(X)$  such that $\sum_{i=1}^m\deg\alpha_i=2n$. 
 Suppose first that   $n<m\le(n+d)$ and let $m=n+e$: thus $0< e\le d$. Choose a point of $X$ and let $\alb_X\colon X\to\Alb(X)$ be the associated Albanese map.  
Let $\theta$ be a a K\"ahler form on $\Alb(X)$: by hypothesis $\dim(\im \alb_X)=d$ and hence there exist holomorphic $1$-forms $\psi_1,\ldots,\psi_e$ on $\Alb(X)$ such that  
\begin{equation}\label{positivo}
\int_{\im(\alb_X)}\psi_1\wedge\ldots\wedge\psi_e\wedge\ov{\psi}_1\wedge\ldots\wedge\ov{\psi}_e\wedge\theta^{d-e}>0.
\end{equation}
For $i=1,\ldots,e$ let $\phi_i:=\alb_X^{*}\psi_i$ and  $\eta:=\alb_X^{*}\theta$. Let $\omega\in H^2_{DR}(X)$ be a K\"ahler class. Equations~\eqref{integrale} and~\eqref{positivo} give that
\begin{equation}
\scriptstyle
\int_{\Gamma^m(X;a)} \pi_1^{*}\phi_1\wedge\ldots\wedge\pi_{e}^{*}\phi_e\wedge 
\pi_{e+1}^{*}\ov{\phi}_1\wedge\ldots\wedge\pi_{2e}^{*}\ov{\phi}_e\wedge 
\pi_{2e+1}^{*}\eta\wedge\ldots\wedge\pi_{e+d}^{*}\eta
\wedge\pi_{e+d+1}^{*}\omega\wedge\ldots\wedge\pi_{m}^{*}\omega
=\int_{X} \phi_1\wedge\ldots\wedge\phi_e\wedge 
\ov{\phi}_1\wedge\ldots\wedge\ov{\phi}_e\wedge 
\eta^{d-e}\wedge\omega^{n-d}>0
\end{equation}
It follows that the homology class of $\Gamma^{m}(X;a)$ is not  torsion. 
Lastly suppose that   $m>(n+d)$.
Let $s$ be given by~\eqref{eccoesse}: then $s\le (m-1)$ because $n>0$. It follows that if $s>0$   the right-hand side of~\eqref{integrale} vanishes (by the binomial formula). Now assume that $s=0$: by~\eqref{integrale} we have that
\begin{equation}\label{lazio}
\int_{\Gamma^m(X;a)} \pi_1^{*}\alpha_1\wedge\ldots\wedge\pi_m^{*}\alpha_m =\int_X \alpha_1\wedge\ldots\wedge\alpha_m.
\end{equation}
Let
\begin{equation}\label{eccoti}
t:=|\{1\le i\le m \mid \deg\alpha_i=1\}|.
\end{equation}
If $t>2d$ then the right-hand side of~\eqref{lazio} vanishes because every class in $H^1_{DR}(X)$ is represented by the pull-back of a closed $1$-form on $\Alb(X)$ via the Albanese map and by hypothesis $\dim(\im \alb_X)=d$. Now suppose that $t\le 2d$. Then 
\begin{equation}
\deg(\pi_1^{*}\alpha_1\wedge\ldots\wedge\pi_m^{*}\alpha_m)\ge t+2(m-t)>2n+2d-t\ge 2n 
\end{equation}
and hence the  right-hand side of~\eqref{lazio} vanishes because the integrand is identically zero. This proves that  if   $m>(n+d)$ the homology class of $\Gamma^{m}(X;a)$ is torsion. 
\end{proof}
\subsection{}\label{subsec:gradofin}
\setcounter{equation}{0}
Let $f\colon X\to Y$ be a map of finite non-zero degree between  projective varieties. Let $a\in X$ and $b:=f(a)$. Then $f_{*}\Gamma^m(X;a)=(\deg f)\Gamma^m(Y;b)$. It follows that if  $\Gamma^m(X;a)\equiv 0$  then  $\Gamma^m(Y;b)\equiv 0$.
\section{Products}\label{sec:prodotti}
We will prove~\Ref{prp}{protokunn} and then we will prove that if $T$ is a complex abelian surface then $\Gamma^5(T;a)\equiv 0$  for any $a\in T$. 
\subsection{Preliminary computations}
\setcounter{equation}{0}
 Let $X$ and $Y$ be projective varieties and $a\in X$, $b\in Y$. 
Let $\es\not=I\subset \{1,\ldots,r\}$ and $\es\not=J\subset \{1,\ldots,s\}$. Thus  $\Delta^r_{I}(X;a)\subset X^r$ and $\Delta^s_{J}(Y;b)\subset Y^s$: we let 
\begin{equation}\label{doppia}
\Delta^{r,s}_{I,J}(X,Y;a,b):=\Delta^r_{I}(X;a)\times \Delta^s_{J}(Y;b)\subset X^r\times Y^s.
\end{equation}
We let $\Delta^{r,s}(X,Y)=\Delta^{r,s}_{\{1,\ldots,r\},\{1,\ldots,s\}}(X,Y;a,b)$. For the remainder of the present section we let
\begin{equation}\label{semplifica}
e:=m+n-1.
\end{equation}
We will constantly make the identification
\begin{equation}\label{scambio}
\begin{matrix}
(X\times Y)^{e} & \overset{\sim}{\lra} & X^{e}\times Y^{e} \\
((x_1,y_1),\ldots,(x_{e},y_{e})) & \mapsto & (x_1,\ldots,x_{e},y_1,\ldots,y_{e})
\end{matrix}
\end{equation}
With the above notation~\Ref{prp}{protokunn} is equivalent to the following rational equivalence:
\begin{equation}\label{diagprod}
\sum\limits_{\es\not=I\subset \{1,\ldots,e\}}(-1)^{e-|I|}\Delta^{e,e}_{I,I}(X,Y;a,b)\equiv 0.
\end{equation}
\begin{prp}\label{prp:delsup}
Let $X$ be a smooth projective variety and $a\in X$. Suppose that $\Gamma^m(X;a)\equiv 0$. Then 
\begin{equation}\label{delsup}
\Delta^{m+r}(X)\equiv \sum_{1\le |J|\le (m-1)}(-1)^{m-1-|J| }{m+r-1-|J| \choose r}\Delta^{m+r}_J(X;a)
\end{equation}
for every $r\ge 0$.
\end{prp}
\begin{proof}
By induction on $r$. If $r=0$ then~\eqref{delsup} is equivalent to $\Gamma^m(X;a)\equiv 0$. Let's prove the inductive step.
Since   $\Gamma^m(X;a)\equiv 0$  we have that 
\begin{multline}
%
\Delta^{m+r+1}(X)\equiv \pi_{1,\ldots,m+r}^{*}\Delta^{m+r}(X)\cdot \pi_{m+r,m+r+1}^{*}\Delta^2(X)\equiv \\
\equiv \pi_{1,\ldots,m+r}^{*}\left(\sum_{\substack{J\subset\{1,\ldots,m+r\} \\ 1\le |J|\le(m-1)}}(-1)^{m-1-|J| }{m+r-1-|J| \choose r}\Delta^{m+r}_J(X;a)\right)\cdot \pi_{m+r,m+r+1}^{*}\Delta^2(X).
\end{multline}
Next notice that
\begin{equation}\label{duecasi}
\pi_{1,\ldots,m+r}^{*}\Delta^{m+r}_J(X;a)\cdot \pi_{m+r,m+r+1}^{*}\Delta^2(X)\equiv
\begin{cases}
\Delta^{m+r+1}_J(X;a) & \text{if $(m+r)\notin J$,}\\
\Delta^{m+r+1}_{J\cup\{m+r+1\}}(X;a) & \text{if $(m+r)\in J$,}
\end{cases}
\end{equation}
Thus $\Delta^{m+r+1}(X)$ is rationally equivalent to a linear combination of cycles  $\Delta^{m+r+1}_J(X;a)$ with $|J|\le(m-1)$ and of cycles $\Delta^{m+r+1}_{K}(X;a)$ where 
\begin{equation}\label{kappacond}
|K|=m,\qquad \{m+r,m+r+1\}\subset K. 
\end{equation}
Let $K$ be such a subset and  write $K=\{i_1,\ldots,i_m\}$ where $i_1<\ldots <i_m$. 
 Let $\iota\colon X^{m}\to X^{m+r+1}$ be the  map which composed  with the $j$-th projection of $X^{m+r+1}$ is equal to the constant map  to $a$ if $j\notin K$, and is equal to the $l$-th projection of $X^{m}$ if if $j=i_l$. Then $\Delta^{m+r+1}_K(X;a)=\iota_{*}\Delta^{m}$ and hence
the equivalence $\Gamma^m(a)\equiv 0$  gives that 
\begin{equation}
\Delta^{m+r+1}_{K}(X;a)\equiv \sum_{\substack{J\subset K \\ 1\le |J|\le(m-1)}}(-1)^{m-1-|J|}\Delta_J^{m+r+1}(X;a).
\end{equation}
Putting everything together we get an equivalence
\begin{equation}\label{caino}
\Delta^{m+r+1}(X)\equiv \sum_{1\le |J|\le (m-1)}(-1)^{m-1-|J| }c_J\Delta^{m+r+1}_J(X;a)
\end{equation}
In order to prove that $c_J={m+r-|J| \choose r+1}$ we distinguish four cases: they are indexed by the intersection
\begin{equation}\label{preiti}
J\cap\{m+r,m+r+1\}.
\end{equation}
Suppose  that~\eqref{preiti} is empty.  We get a contribution (to $c_J$) of ${m+r-1-|J| \choose r}$ from the first case in~\eqref{duecasi}, and a contribution of 
\begin{equation}
\scriptstyle
|\{(J\cup\{m+r,m+r+1\})\subset K \subset\{1,\ldots,m+r+1\} \mid |K|=m\}|=
{m+r-1-|J|\choose m-2-|J|}={m+r-1-|J|\choose r+1}
\end{equation}
from the subsets $K$ satisfying~\eqref{kappacond}. This proves that $c_J={m+r-|J| \choose r+1}$  in this case. The proof in the other three cases is similar.
\end{proof}
\begin{crl}\label{crl:delsup}
Let $X$ be a smooth projective variety and $a\in X$. Suppose that $\Gamma^m(X;a)\equiv 0$. Let $s\ge 0$ and  $I\subset \{1,\ldots,m+s\}$ be a subset of cardinality at least
 $m$.  Then 
\begin{equation}\label{delsupdue}
\Delta^{m+s}_I(X;a)\equiv \sum_{\substack{J\subset I \\ 1\le |J|\le(m-1)}}(-1)^{m-1-|J|}{|I|-|J| -1\choose |I|-m}\Delta^{m+s}_J(X;a).
\end{equation}
\end{crl}
\begin{proof}
Let $q:=|I|$ and $I=\{i_1,\ldots,i_q\}$ where $i_1<\ldots <i_q$. Let $\iota\colon X^{q}\to X^{m+s}$ be the  map which composed  with the $j$-th projection of  $X^{m+s}$ is equal to the constant map to $a$ if $j\notin I$, and   is equal to the $l$-th projection of $X^{m}$ if if $j=i_l$. 
 Then $\Delta^{m+s}_I(X;a)=\iota_{*}\Delta^{q}(X)$ and one gets~\eqref{delsupdue} by invoking~\Ref{prp}{delsup}. 
\end{proof}
\begin{crl}\label{crl:emmabon}
Let $X$, $Y$ be  smooth projective varieties and $a\in X$, $b\in Y$. Suppose that $\Gamma^m(X;a)\equiv 0$ and   $\Gamma^n(Y;a)\equiv 0$. 
Assume that $m\le n$. Let $I\subset \{1,\ldots,e\}$ (recall that $e=m+n-1$). 
\begin{enumerate}
\item 
 If $n\le |I|$  then 
\begin{equation}\label{moavero}
\Delta^{e,e}_{I,I}(X,Y;a,b)\equiv \sum_{\substack{(J\cup K)\subset I \\ 1\le |J|\le(m-1) \\ 1\le |K|\le (n-1)}}(-1)^{m+n-|J|-|K|}
{|I|-|J| -1\choose m-|J|-1}{|I|-|K|-1\choose n-|K| -1}\Delta^{e,e}_{J,K}(X,Y;a,b).
\end{equation}
\item 
If $m\le |I|<n$  then 
\begin{equation}\label{saccomanni}
\Delta^{e,e}_{I,I}(X,Y;a,b)\equiv \sum_{\substack{J\subset  I \\ 1\le |J|\le(m-1)}}(-1)^{m-1-|J|}
{|I|-|J| -1\choose m-|J|-1}\Delta^{e,e}_{J,I}(X,Y;a,b).
\end{equation}
\end{enumerate}
\end{crl}
\begin{proof}
By  definition $\Delta^{e,e}_{I,I}(X,Y;a,b)=\Delta^{e}_{I}(X;a)\times \Delta^{e}_{I}(Y;b)$. Now suppose that   $n\le |I|$. By~\Ref{crl}{delsup}  the first factor is rationally equivalent to a linear combination of $\Delta_J^e(X;a)$'s with $J\subset I$ and $1\le |J|\le(m-1)$,  the second factor is rationally equivalent to a linear combination of $\Delta_K^e(Y;b)$'s with $K\subset I$ and $1\le |K|\le(n-1)$: writing out the product one gets~\eqref{moavero}. The proof of~\eqref{saccomanni} is similar.
\end{proof}
\subsection{Linear relations between binomial coefficients.}\label{subsec:coeffbin}
\setcounter{equation}{0}
The following fact will be useful:
\begin{equation}\label{combcomb}
\sum_{t=0}^n(-1)^t p(t){n\choose t}=0\qquad \forall p\in\QQ[x]\text{ such that $\deg p<n$.}
\end{equation}
In order to prove~\eqref{combcomb} let $d<n$: then we have
\begin{equation}
\sum_{t=0}^n(-1)^t {t\choose d}{n\choose t}={n\choose d}\sum_{t=d}^n(-1)^t {n-d\choose t-d}=(1-1)^{n-d}=0.
\end{equation}
Since $\{{x\choose 0},{x\choose 1},\ldots,{x\choose n-1}\}$ is a basis of the vector space of polynomials of degree at most $(n-1)$ Equation~\eqref{combcomb} follows.
\subsection{Proof of the main result.}\label{subsec:combinatorica}
\setcounter{equation}{0}
 We will prove~\Ref{prp}{protokunn}. 
 As noticed above it suffices to prove that~\eqref{diagprod} holds. 
 Without loss of generality we may assume that $m\le n$. \Ref{crl}{emmabon} gives that for each   $1\le t\le e$  and $J,K\subset \{1,\ldots,e\}$  with $|J|\le(m-1)$, $|K|\le(n-1)$ there exists $c_{J,K}(t)$  such that
\begin{equation}\label{}
\sum_{|I|=t}\Delta^{e,e}_{I,I}(X,Y;a,b)\equiv \sum_{\substack{J,K\subset  \{1,\ldots,e\} \\ 1\le |J|\le(m-1) \\ 1\le |K|\le (n-1)}}
c_{J,K}(t)\Delta^{e,e}_{J,K}(X,Y;a,b).
\end{equation}
It will suffice to prove that for each $J,K$ as above we have
\begin{equation}\label{jannacci}
\sum_{t=1}^{e}(-1)^{t} c_{J,K}(t)=0.
\end{equation}
Equations~\eqref{moavero} and~\eqref{saccomanni} give that  $c_{J,K}(t)=0$ if $t<|J\cup K|$  and that 
\begin{equation}\label{cikappati}
c_{J,K}(t)= (-1)^{m+n-|J|-|K|}
{t-|J| -1\choose m-|J|-1}{t-|K|-1\choose n-|K|-1}{e-|J\cup K| \choose t-|J\cup K| },\quad \max\{|J\cup K|,n\}\le t\le e.
\end{equation}
We distinguish between the four cases:
\begin{enumerate}
\item 
$J\not\subset K$.
\item 
$J\subset K$ and $m\le|K|$.
\item 
$J\subset K$, $J\not=K$ and $|K|<m$.
\item 
$J= K$ and $|K|<m$.
\end{enumerate}
Suppose that~(1) holds.
 Then~\Ref{crl}{emmabon} gives that  $c_{J,K}(t)=0$ if $t<n$.  Let $p\in\QQ[x]$ be given by 
 \begin{equation}\label{bersani}
p:=(-1)^{m+n-|J|-|K|}{x-|J|-1\choose m-|J|-1}{x-|K|-1\choose n-|K|-1}.
\end{equation}
We must prove that
\begin{equation}\label{serenella}
\sum_{t=\max\{|J\cup K|,n\}}^e (-1)^t p(t){e-|J\cup K| \choose t-|J\cup K| }=0.
\end{equation}
If $n\le |J\cup K|$ then~\eqref{serenella} follows at once from~\eqref{combcomb} (notice that $\deg p<(e-|J\cup K|)$),   
 if $n< |J\cup K|$ then~\eqref{serenella} follows  from~\eqref{combcomb} and the fact that $p(i)=0$ for $|J\cup K|\le i\le(n-1)$. This proves~\eqref{jannacci} if Item~(1) above holds. Now let's assume that Item~(2) above holds. Then $|J\cup K|=|K|<n$: it follows that  if $n\le t$ then $c_{J,K}(t)$  is given by~\eqref{cikappati}. On the other hand~\Ref{crl}{emmabon} gives that  if $t<n$ and $t\not=|K|$ then $c_{J,K}(t)=0$, and
 \begin{equation}
c_{J,K}(|K|)=(-1)^{m-1-|J|}{|K|-|J|-1\choose m-|J|-1}.
\end{equation}
 Thus we must prove that
\begin{equation}\label{marolla}
(-1)^{|K|}(-1)^{m-1-|J|}{|K|-|J|-1\choose m-|J|-1}+\sum_{t=n}^e (-1)^t p(t){e-|J\cup K| \choose t-|J\cup K| }=0
\end{equation}
where $p$ is given by~\eqref{bersani}. Now notice that $0=p(|K|+1)=\ldots=p(n-1)$: thus~\eqref{combcomb} gives that 
\begin{equation*}
\scriptstyle
\sum_{t=n}^e (-1)^t p(t){e-|J\cup K| \choose t-|J\cup K| }=-(-1)^{|K|}p(|K|){e-|K|\choose 0}=(-1)^{m+n-1-|J|}
{|K|-|J|-1\choose m-|J|-1}{-1\choose n-|K|-1}=(-1)^{m-|J|-|K|}{|K|-|J|-1\choose m-|J|-1}.
\end{equation*}
This proves that~\eqref{marolla} holds. If Item~(3) above holds one proves~\eqref{jannacci} arguing as in Item~(1), if Item~(4) holds the argument is similar to that given if Item~(2) holds. 
\qed
\subsection{Stability.}\label{subsec:stabilizza}
\setcounter{equation}{0}
We will prove a result that will be useful later on.
\begin{prp}\label{prp:stabile}
Let $X$ be a smooth projective variety and $a\in X$. Suppose that $\Gamma^m(X;a)\equiv 0$. If $s\ge 0$ then 
$\Gamma^{m+s}(X;a)\equiv 0$.
\end{prp}
\begin{proof}
If $\dim X=0$ the result is trivial. Assume that $\dim X>0$.  By definition
\begin{equation}
\Gamma^{m+s}(X;a):=\sum\limits_{\es\not= I\subset \{1,2,\ldots,m+s\}}(-1)^{m+s-|I|}\Delta^{m+s}_I(X;a).
\end{equation}
Replacing $\Delta^{m+s}_I(X;a)$ for $m\le |I|\le(m+s)$  by the right-hand side of~\eqref{delsupdue} we get that
\begin{equation}
\Gamma^{m+s}(X;a):=\sum_{1\le\ell\le(m-1)}c_{\ell}
\left(\sum\limits_{| I |=\ell}(-1)\Delta^{m+s}_I(X;a)\right)
\end{equation}
where
\begin{equation}
c_{\ell}=\sum_{r=0}^s(-1)^{m-\ell-1+s-r}{m-\ell-1+r\choose m-\ell-1}{m+s-\ell\choose s-r}+(-1)^{m+s-\ell}.
\end{equation}
Thus it suffices to prove that $c_{\ell}=0$ for $1\le \ell\le(m-1)$. Letting $t=s-r$ we get that 
\begin{multline}
(-1)^{m-\ell-1}c_{\ell}=\sum_{t=0}^s(-1)^{t}{m-\ell-1+s-t\choose m-\ell-1}{m+s-\ell\choose t}+(-1)^{s-1}=\\
=\sum_{t=0}^{m+s-\ell}(-1)^{t}{m-\ell-1+s-t\choose m-\ell-1}{m+s-\ell\choose t}=0
\end{multline}
where the last equality follows from~\eqref{combcomb}.
\qed
\end{proof}
\subsection{Applications}\label{subsec:supab}
\setcounter{equation}{0}
\begin{prp}\label{prp:jaciper}
Suppose that $C$ is a smooth projective curve of genus $g$ and  that there exists a degree-$2$ map $f\colon C\to\PP^1$ ramified at $p\in C$.  Then 
\begin{enumerate}
\item 
$\Gamma^{2g+1}(C^g;(p,\ldots,p))\equiv 0$,
\item 
$\Gamma^{2g+1}(C^{(g)};gp)\equiv 0$, and
\item 
$\Gamma^{2g+1}(\Pic^0(C);a)\equiv 0$ for any $a\in \Pic^0(C)$. 
\end{enumerate}
\end{prp}
\begin{proof}
By Proposition~4.8 of~\cite{groscho} we have $\Gamma^3(C;p)\equiv 0$. Repeated application of~\Ref{prp}{protokunn} gives the first item.  The  quotient map $C^g\to C^{(g)}$ is finite and the image of $(p,\ldots,p)$ is $gp$: 
thus Item~(2)  follows from Item~(1) and~\Ref{subsec}{gradofin}.
Let $u_g\colon C^{(g)}\to \Pic^0(C)$ be the map $D\mapsto [D-gp]$: since $u_g$ is birational Item~(2) and~\Ref{subsec}{gradofin} give that  $\Gamma^{2g+1}(\Pic^0(C);{\bf 0})\equiv 0$ where ${\bf 0}$ is the origin of $\Pic^0(C)$. Acting by translations we get that $\Gamma^{2g+1}(\Pic^0(C);a)\equiv 0$ for any $a\in \Pic^0(C)$. 
\end{proof}
\begin{crl}\label{crl:jaciper}
If $T$ is a complex abelian surface then $\Gamma^5(T;a)\equiv 0$ for any $a\in T$.
\end{crl}
\begin{proof}
  There exists a principally polarized abelian surface $J$ and an isogeny $J\to T$. By~\Ref{subsec}{gradofin} it suffices to prove that $\Gamma^5(J;b)\equiv 0$ for any $b\in J$. The surface  $J$ is either a  product of two elliptic curves $E_1,E_2$ or   
 the Jacobian of a smooth genus-$2$ curve $C$.  Suppose that the former holds.  Let $a=(p_1,p_2)$ where $p_i\in E_i$ for $i=1,2$. Then $\Gamma^3(E_i;p_i)\equiv 0$  by Proposition~4.8 of~\cite{groscho}  and hence~\Ref{prp}{protokunn} gives that $\Gamma^5(E_1\times E_2;(p_1,p_2))\equiv 0$. If $J$ is    
 the Jacobian of a smooth genus-$2$ curve $C$ the corollary follows at once from~\Ref{prp}{jaciper}. 
\end{proof}
\section{$\PP^r$-fibrations}\label{sec:fibrazioni}
Let $Y$ be a smooth projective variety. Let $\cF$ be a locally-free sheaf of rank $(r+1)$ on $Y$ and $X:=\PP(\cF)$. Thus the structure map $\rho\colon X\to Y$ is a $\PP^r$-fibration. Let  $Z:=c_1(\cO_X(1))\in\CH^1(X)$. Suppose that there exists $b\in Y$ such that $\Gamma^m(Y;b)\equiv 0$ and let $a\in\rho^{-1}(b)$. If $\PP(\cF)$ is trivial  then 
\begin{equation}\label{traslo}
\Gamma^{m+r}(X;a)\equiv 0
\end{equation}
 by~\Ref{expl}{spapro} and~\Ref{prp}{protokunn}. In general~\eqref{traslo} does not hold.  In fact suppose that
 $Y$ is a $K3$ surface and hence  $\Gamma^3(Y;b)\equiv 0$  where $b$ is a point lying on a rational curve~\cite{beauvoisin}.  If $\Gamma^{3+r}(X;a)\equiv 0$ then the top self-intersection of any divisor class on $X$ is a multiple of $[a]$, see~\Ref{subsec}{significato}: considering $Z^{r+2}$ we get that   $c_2(\cF)$ is a multiple of $[b]$. 
We will prove the following results.
\begin{prp}\label{prp:rigcurva}
Keep notation as above and suppose that $\dim Y=1$. If $\Gamma^m(Y;b)\equiv 0$ then  
$\Gamma^{m+r}(X;a)\equiv 0$.
\end{prp}
\begin{prp}\label{prp:rigsuperficie}
Keep notation as above and suppose that $\dim Y=2$. If $\Gamma^{m-1}(Y;b)\equiv 0$, or $\Gamma^{m}(Y;b)\equiv 0$ and both $c_1(\cF)^2$, $c_2(\cF)$ are multiples of $[b]$, then  $\Gamma^{m+r}(X;a)\equiv 0$ .
\end{prp}
As an appplication we will prove the following.
\begin{prp}\label{prp:prodsim}
Suppose that $C$ is a smooth projective curve of genus  $g\le 2$ over an algebraically closed field $\KK$ and that $p\in C$ is such that $\dim|\cO_C(2p)| \ge 1$.  Then $\Gamma^{d+g+1}(C^{(d)};dp)\equiv 0$ for any $d\ge 0$.
\end{prp}
\subsection{Comparing diagonals}
\setcounter{equation}{0}
Let  $\rho^n\colon X^n\to Y^n$ be the $n$-th cartesian product of $\rho$. Let $\pi_i\colon X^n\to X$ be the $i$-th projection and $Z_i:=\pi_i^{*}Z$. Given a multi-index $E=(e_1,\ldots,e_n)$ with $0\le e_i$ for $1\le i\le n$ we let $Z^E:=Z_1^{e_1}\cdot\ldots\cdot Z^{e_n}_n$. We let 
\begin{equation}
\max E:=\max\{e_1,\ldots,e_n\},\qquad |E|:=e_1+\ldots+e_n.
\end{equation}
Let $d:=\dim Y$ and $[\Delta^n(X)]\in\CH_{d+r}(X^n)$ be the class of the (smallest) diagonal. Since $\rho^n$ is a $(\PP^r)^n$-fibration we may write
\begin{equation}\label{darisolvere}
[\Delta^n(X)]=\sum_{\max E\le r}(\rho^n)^{*}(w_E(\cF))\cdot Z^E,\qquad w_E(\cF)\in \CH_{|E|+d-r(n-1)}(Y^n).
\end{equation}
In order to describe the classes $w_E$ we let $\delta^n_Y\colon Y\hra Y^n$ and  $\delta^n_X\colon X\hra X^n$ be the diagonal embeddings. 
\begin{prp}\label{prp:coefficienti}
Let $r\ge 0$ and  $E=(e_1,\ldots,e_n)$ be a multi-index. There exists  a universal  polynomial $P_E\in\QQ[x_1,\ldots,x_q]$, where $q:=(r(n-1)-|E|)$,  such that the following holds. Let $\cF$ be a locally-free sheaf of rank $(r+1)$ on $Y$: then (notation as above)   $w_E(\cF)=\delta^n_{Y,*}(P_E(c_1(\cF),\ldots,c_q(\cF))$. 
\end{prp}
\begin{proof}
Let  $s_i(\cF)$ be the $i$-th Segre class of $\cF$ and $E^{\vee}:=(r-e_1,\ldots,r-e_n)$. Then
\begin{equation}\label{kyenge}
\rho^n_{*}([\Delta^n(X)]\cdot  Z^{E^{\vee}})=\delta_{Y,*}^n(s_{|E^{\vee}|-r}(\cF)).
\end{equation}
(By convention $s_{i}(\cF)=0$ if $i<0$.)
On the other hand  let $J=(j_1,\ldots,j_n)$ be a multi-index: then
\begin{equation}\label{ogbonna}
\rho^n_{*}\left(\left(\sum_{\max H\le r}(\rho^n)^{*}(w_H(\cF))\cdot Z^H\right)\cdot Z^{J}\right)=
\sum_{\max H\le r} w_{H}(\cF)\cdot \pi_1^{*}(s_{h_1+j_1-r})\cdot\ldots\cdot \pi_n^{*}(s_{h_n+j_n-r}).
\end{equation}
Equations~\eqref{kyenge} and~\eqref{ogbonna} give that
\begin{multline}\label{esame}
\delta_{Y,*}^n(s_{|E^{\vee}|-r}(\cF))=\rho^n_{*}([\Delta^n(X)]\cdot  Z^{E^{\vee}})=
\sum_{\max H\le r} w_{H}(\cF)\cdot \pi_1^{*}(s_{h_1-e_1})\cdot\ldots\cdot \pi_n^{*}(s_{h_n-e_n})= \\
=w_{E}(\cF)+\sum_{\substack{|H|>|E| \\ r\ge \max H}} w_{H}(\cF)\cdot \pi_1^{*}(s_{h_1-e_1})\cdot\ldots\cdot \pi_n^{*}(s_{h_n-e_n}).
\end{multline}
Starting from the highest possible value of $|E|$ i.e.~$rn$ and going through descending values of $|E|$ one gets the proposition. 
\end{proof}
\begin{rmk}\label{rmk:esplicito}
The proof of~\Ref{prp}{coefficienti} gives an iterative algorithm for the computation of $w_E(\cF)$. A straightforward computation gives the   formulae
\begin{equation*}\label{esplicito}
w_E(\cF)=
\begin{cases}
0 & \text{if $|E|>r(n-1)$}, \\
[\Delta^n(Y)] & \text{if $|E|=r(n-1)$}, \\
(\lambda_E(1)-1)\delta^n_{Y,*}( c_1(\cF))& \text{if $|E|=r(n-1)-1$}, \\
\frac{1}{2}(\lambda_E(1) -1)(\lambda_E(1) -2)\delta^n_{Y,*}( c_1(\cF)^2) + 
(\lambda_E(2) -1 ) \delta^n_{Y,*}(c_2(\cF))& \text{if $|E|=r(n-1)-2$},
\end{cases}
\end{equation*}
where 
\begin{equation}
\lambda_E(p):=|\{ 1\le i\le n \mid  e_i+p\le r \}|. 
\end{equation}
\end{rmk}
\subsection{Comparing modified diagonals}
\setcounter{equation}{0}
We  will compare $\Gamma^{m+r}(X;a)$ and $\Gamma^{m+r}(Y;b)$. In the present subsection  
 $\es\not=I\subset\{1,\ldots,m+r\}$ and $I^c:=(\{1,\ldots,m+r\}\setminus I)$; we let $\pi_I\colon X^{m+r}\to X^{|I|}$ be the projection determined by $I$. We also let $H=(h_1,\ldots,h_{m+r})$ be a multi-index. If  $\max H\le r$ we let $\Top H:=\{1\le i\le n \mid h_i=r\}$.  Applying~\Ref{prp}{coefficienti} and~\Ref{rmk}{esplicito}  we get that
\begin{multline}\label{barvitali}
\Delta^{m+r}_I(X;a)=
(\rho^{m+r})^{*}(\Delta^{m+r}_I(Y;{b}))\cdot  \sum_{\substack{\max H\le r \\ |H|=r(m+r-1) \\ I^c\subset\Top H }} Z^H+ \\ 
+(\rho^{m+r})^{*}\left(\pi_I^{*}\delta^{|I|}_{Y,*}(c_1(\cF))\times\pi_{I^c}^{*}(\underbrace{{b}\times\ldots\times{b}}_{|I^c|})\right)
\cdot\sum_{\substack{\max H\le r \\ |H|=r(m+r-1)-1 \\  I^c\subset\Top H }} (\lambda_H(1)-1) Z^H+ \\
+(\rho^{m+r})^{*}\left(\pi_I^{*}\delta^{|I|}_{Y,*}(c_1(\cF)^2)\times\pi_{I^c}^{*}(\underbrace{{b}\times\ldots\times{b}}_{|I^c|})\right)
\cdot\sum_{\substack{\max H\le r \\ |H|=r(m+r-1)-2  \\ I^c\subset\Top H}} \frac{1}{2}(\lambda_H(1)-1)(\lambda_H(1)-2) Z^H+\\
+(\rho^{m+r})^{*}\left(\pi_I^{*}\delta^{|I|}_{Y,*}(c_2(\cF))\times\pi_{I^c}^{*}(\underbrace{{b}\times\ldots\times{b}}_{|I^c|})\right)
\cdot\sum_{\substack{\max H\le r \\ |H|=r(n-1)-2   \\ I^c\subset\Top H}} (\lambda_H(2)-1) Z^H+\cR
\end{multline}
where
\begin{equation}\label{brasile}
\cR=\sum_{\substack{\max H\le r \\ |H|<r(n-1)-2}}Q_H Z^H
\end{equation}
and each $Q_H$ appearing in~\eqref{brasile} vanishes if the Chern classes of $\cF$ of degree higher than $2$ are zero. 
It follows that 
\begin{multline}\label{barsport}
\Gamma^{m+r}(X;a)=
\sum_{\substack{\max H\le r \\ |H|=r(m+r-1)}}(\rho^{m+r})^{*}\left(\sum_{I^c\subset\Top H}(-1)^{m+r-|I|}
\Delta^{m+r}_I(Y;{b})\right)\cdot   Z^H+ \\ 
+\sum_{\substack{\max H\le r \\ |H|=r(m+r-1)-1}}(\rho^{m+r})^{*}\left(\sum_{I^c\subset\Top H}(-1)^{m+r-|I|}
\left(\pi_I^{*}\delta^{|I|}_{Y,*}(c_1(\cF))\times\pi_{I^c}^{*}(\underbrace{{b}\times\ldots\times{b}}_{|I^c|})\right)\right)
\cdot \epsilon_H Z^H+ \\
+\sum_{\substack{\max H\le r \\ |H|=r(m+r-1)-2}}(\rho^{m+r})^{*}\left(\sum_{I^c\subset\Top H}(-1)^{m+r-|I|}
\left(\pi_I^{*}\delta^{|I|}_{Y,*}(c_1(\cF)^2)\times\pi_{I^c}^{*}(\underbrace{{b}\times\ldots\times{b}}_{|I^c|})\right)\right)
\cdot \mu_H Z^H+\\
+\sum_{\substack{\max H\le r \\ |H|=r(m+r-1)-2}}(\rho^{m+r})^{*}\left(\sum_{I^c\subset\Top H}(-1)^{m+r-|I|}
\left(\pi_I^{*}\delta^{|I|}_{Y,*}(c_2(\cF))\times\pi_{I^c}^{*}(\underbrace{{b}\times\ldots\times{b}}_{|I^c|})\right)\right)
\cdot \nu_H Z^H+\cT
\end{multline}
where $\epsilon_H:=(\lambda_H(1)-1)$, $\mu_H:=(\lambda_H(1)-1)(\lambda_H(1)-2)/2$, $\nu_H:=(\lambda_H(2)-1)$, and $\cT$ has an expansion similar to that of $\cR$, see~\eqref{brasile} and the comment following it.
\begin{rmk}\label{rmk:svanisce}
Suppose that $\Gamma^m(Y;b)=0$. Then the  first addend on the right-hand side of~\eqref{barsport} vanishes. In fact it is clearly independent of the rank-$r$ locally-free sheaf $\cF$ and it is $0$ for trivial $\cF$ by~\Ref{prp}{protokunn}: it follows that it vanishes.
\end{rmk}
\subsection{$\PP^r$-bundles over curves}\label{subsec:rigcurva}
\setcounter{equation}{0}
We will prove~\Ref{prp}{rigcurva}. We start with an auxiliary result.
\begin{clm}\label{clm:sommalt}
Let $Y$ be a smooth projective variety and $b\in Y$. Suppose that $\Gamma^m(Y;b)=0$. Let $\gz\in\CH(Y)$: then
\begin{equation}\label{sommalt}
\sum_{I\subset\{1,\ldots,(m-1)\}}(-1)^{|I|}\pi_I^{*}\delta_{Y,*}^{|I|}(\gz)\times\pi_{I^c}^{*}(\underbrace{b,\ldots,b}_{|I^c|})=0.
\end{equation}
\end{clm}
\begin{proof}
Let $\pi_{\{1,\ldots,(m-1)\}}\colon Y^m\to Y^{m-1}$ be the projection to the first $(m-1)$ coordinates. Then
\begin{equation}\label{sgargiante}
\pi_{\{1,\ldots,(m-1)\},*}(\Gamma^m(Y;b)\cdot\pi_m^{*}\gz)=0.
\end{equation}
The claim follows because the left-hand side of~\eqref{sgargiante} equals the left-hand side of~\eqref{sommalt} multiplied by $(-1)^m$.
\end{proof}
By~\eqref{barsport} and~\Ref{rmk}{svanisce} we must prove that if $H=(h_1,\ldots,h_{m+r})$ is a multi-index such that $\max H\le r$ and $|H|=r(m+r-1)-1$  then
\begin{equation}\label{mammamia}
\sum_{I^c\subset\Top H}(-1)^{m+r-|I|}
\pi_I^{*}\delta^{|I|}_{Y,*}(c_1(\cF))\times\pi_{I^c}^{*}(\underbrace{b,\ldots,b}_{|I^c|})=0.
\end{equation}
A straightforward computation shows that $|\Top H|\ge(m-1)$: thus~\eqref{mammamia} holds by~\Ref{clm}{sommalt}. 
\qed
\subsection{$\PP^r$-bundles over surfaces}
\setcounter{equation}{0}
We will prove~\Ref{prp}{rigsuperficie}. Notice that $\Gamma^m(Y;b)=0$: in fact  it holds either by hypothesis or by~\Ref{prp}{stabile} if  $\Gamma^{m-1}(Y;b)=0$.  
Moreover~\eqref{mammamia} holds in this case as well, the argument is that given in~\Ref{subsec}{rigcurva}. Thus~\eqref{barsport} and~\Ref{rmk}{svanisce} give that we must prove the following: if $H=(h_1,\ldots,h_{m+r})$ is a multi-index such that $\max H\le r$ and $|H|=r(m+r-1)-2$  then
\begin{equation}\label{papas}
\sum_{I^c\subset\Top H}(-1)^{m+r-|I|}
\left(\pi_I^{*}\delta^{|I|}_{Y,*}(\mu_H c_1(\cF)^2+\nu_H c_2(\cF))\times\pi_{I^c}^{*}(\underbrace{b,\ldots,b}_{|I^c|})\right)=0.
\end{equation}
A straightforward computation shows that $|\Top H|\ge(m-2)$ and that equality holds if and only if $(r-1)\le h_i\le r$ for all $1\le i\le (m+r)$ (and thus the set of indices $i$ such that $h_i=(r-1)$ has cardinality $(r+2)$). If $\Gamma^{m-1}(Y;b)=0$ then~\eqref{papas} holds by~\Ref{clm}{sommalt}. If  both $c_1(\cF)^2$, $c_2(\cF)$ are multiples of $b$ then each term in the summation in the left-hand side of~\eqref{papas} 
is a multiple of $b$ and the coefficients sum up to $0$.  
\qed
\subsection{Symmetric products of curves}
\setcounter{equation}{0}
If the genus of $C$ is $0$ then $C^{(d)}\cong\PP^d$ and hence the result holds trivially, see~\Ref{expl}{spapro}. Suppose that the genus of $C$ is $1$. If $d=1$ then $\Gamma^3(C;p)\equiv 0$ by~\cite{groscho}. Let $d>1$ and let $u_d\colon C^{(d)}\to \Pic^0(C)$ be the map sending $D$ to $[D-d p]$. Since $u_d$ is $\PP^{d-1}$-fibration we get that $\Gamma^{d+2}(C;dp)\equiv 0$ by~\Ref{prp}{prodsim} and the equivalence $\Gamma^3(C;p)\equiv 0$. Lastly suppose that  the genus of $C$ is $2$. If $d=1$ then $\Gamma^3(C;p)\equiv 0$ by~\cite{groscho} and if $d=2$ then $\Gamma^5(C^{(2)};2p)\equiv 0$ by~\Ref{prp}{jaciper}. Now assume that $d>2$ and let $u_d\colon C^{(d)}\to \Pic^0(C)$ be the map sending $D$ to $[D-d p]$. Then $u_d$ is $\PP^{d-2}$-fibration and we may write $C^{(d)}\cong\PP(\cE_d)$ where $\cE_d$ is a locally-free sheaf on $\Pic^0(C)$ such that 
\begin{equation}
c_1(\cE_d)=-[\{[x-p] \mid x\in C\}],\qquad c_2(\cE_d)=[{\bf 0}],
\end{equation}
see Example~4.3.3 of~\cite{fulton}. By~\Ref{prp}{jaciper}  we have  $\Gamma^5(J(C);{\bf 0})\equiv 0$; since $c_1(\cE_d)^2=2[{\bf 0}]$ we get that $\Gamma^{d+2}(C^{(d)};dp)\equiv 0$ by~\Ref{prp}{prodsim}.
\section{Blow-ups}\label{sec:scoppio}
We will prove~\Ref{prp}{blowdel}.
A comment regarding the hypotheses of~\Ref{prp}{blowdel}. 
Let $Y$ be a complex $K3$ surface  and $X\to Y$ be the blow-up of $y\in Y$. We know (Beauville and Voisin) that there exists $c\in Y$ such that $\Gamma^3(Y;c)\equiv 0$, but if $y$ is not rationally equivalent to $c$ then there exists no $a\in X$ such that 
$\Gamma^3(X;a)\equiv 0$, this follows from~\Ref{rmk}{unicopunto}. If  $e=0,1$ then~\Ref{prp}{blowdel} is trivial, hence we will assume that $e\ge 2$. We let $f\colon X\to Y$ be the blow-up of $V$ and $E\subset X$  the exceptional divisor of $f$. Thus $a\in E$. Let $g\colon E\to V$ be defined by the restriction of $f$ to $E$, and $(E/V)^t$ be the $t$-th fibered product of $g\colon E\to V$. Let $(E/V)^{t}$ be the $t$-th fibered product of $g\colon E\to V$. The following commutative diagram will play a r\^ole in the proof of~\Ref{prp}{blowdel}
\begin{equation}
\xymatrix{ (E/V)^{t} \ar_{\gamma_{t}}[r] \ar@/{}_{-1pc}/^{\alpha_{t}}[urrd] \ar_{}[d] & E^{t} \ar_{\beta_{t}}[r] \ar_{g^{t}}[d] & X^{t} \ar_{f^{t}}[d] \\
\Delta^{t}(V) \ar^{}[r]  & V^{t} \ar^{}[r] & Y^{t} } 
\end{equation}
(The maps which haven't  been defined are the natural ones.) Whenever there is no danger of confusion we denote  $\alpha_{t}((E/V)^{t})$ by $(E/V)^{t}$.
\subsection{Pull-back of the modified diagonal.}
\setcounter{equation}{0}
 On $E$ we have an exact sequence of locally-free sheaves:
\begin{equation}
0\lra\cO_E(-1)\lra g^{*}N_{V/Y}\lra Q\lra 0.
\end{equation}
For $i=1,\ldots,t$ let $Q_i(t)$ be the pull-back of $Q$ to $E^{t}$ via the $i$-th projection $E^{t}\to E$: thus $Q_i(t)$ is locally-free of rank $(e-1)$. 
\begin{prp}\label{prp:eccesso}
Keep notation as above and let $d({t}):=({t}-1)(e-1)-1$. We have the following equalities  in $\CH_{\dim X}(X^{t})$:
\begin{equation}\label{eccesso}
(f^{t})^{*}\Delta^{t}(Y)=
\begin{cases}
\Delta^{t}(X) & \text{if ${t}=1$,} \\
\Delta^{t}(X)+\beta_{{t},*}((g^{t})^{*}(\Delta^{t}(V))\cdot c_{d({t})}(\oplus_{j=1}^{t} Q_j({t}))) & \text{if ${t}> 1$.}
\end{cases}
\end{equation}
\end{prp}
\begin{proof}
The equality of schemes $f^{-1}\Delta^1(Y)=\Delta^1(X)$ gives~\eqref{eccesso}  for ${t}=1$. Now let's assume that ${t}>1$. The closed set $(f^{t})^{-1}\Delta^{t}(Y)$ has the following decomposition into irreducible components: 
\begin{equation}
(f^{t})^{-1}\Delta^{t}(Y)=\Delta^{t}(X)\cup (E/V)^{t}. 
\end{equation}
The  dimension  of  $(E/V)^{t}$ is equal to $(\dim X+({t}-1)(e-1)-1)$ and hence  is larger than the expected dimension unless unless $2={t}=e$. It follows  that if  ${t}=2$ and $e=2$ then $(f^2)^{*}\Delta^2(Y)=a\Delta^2(X)+b (E/V)^2$:  one checks easily that $1=a=b$ and hence~\eqref{eccesso}  holds if ${t}=2$ and $e=2$. Now suppose that that ${t}>1$ and $({t},e)\not=(2,2)$. Let $U:=(X^{t}\setminus(\Delta^{t}(X)\cap (E/V)^{t}))$ and $\cZ:= (E/V)^{t} \cap U=(E/V)^{t} \setminus \Delta^{t}(X)$. Notice that $(E/V)^{t}$ is smooth and hence  the open subset $\cZ$ is smooth as well. Let $\iota\colon \cZ\hra U$ be the inclusion. The restriction of $(f^{t})^{*}\Delta^{t}(Y)$ to $U$ is equal to 
\begin{equation}
[\Delta^{t}(X)\cap U]+\iota_{*}(c_d({t})(\cN))
\end{equation}
where $\cN$ is the obstruction bundle (see~\cite{fulton}, Cor.~8.1.2 and Prop.~6.1(a)). One easily identifies $\cN$ with the restriction of $\oplus_{j=1}^{t} Q_j({t})$ to $\cZ$. It follows that the restrictions to $U$ of the left and right hand sides of~\eqref{eccesso} are equal. The proposition follows because the  dimension of $(X^{t}\setminus U)=\Delta^{t}(X)\cap (E/V)^{t}$ is equal to $(\dim X-1)$, which is strictly smaller than $\dim X$.
\end{proof}
\begin{crl}\label{crl:eccesso}
Keep notation and assumptions as above. Let $I\subset\{1,\ldots,(n+1)\}$ be non-empty and $I^c:=(\{1,\ldots,(n+1)\}\setminus I)$.  Let $Q_j$ denote $Q_j(n+1)$ and let ${t}:=| I |$. Then
\begin{equation}
\scriptstyle
(f^{n+1})^{*}\Delta_I(Y;b)=
\begin{cases}
\scriptstyle \Delta_I(X;a) & \scriptstyle \text{if $| I |=1$,} \\
\scriptstyle \Delta_I(X;a)+\beta_{n+1,*}((g^{n+1})^{*}\Delta_I(V;b)\cdot c_{d({t})}\left(\bigoplus\limits_{j\in I} Q_j\right)\cdot \prod_{j\in I^c} c_{e-1}(Q_j)) & \scriptstyle\text{if $| I |> 1$.}
\end{cases}
\end{equation}
\end{crl}
\begin{proof}
For $1\le i\le(n+1)$  let $\rho_{i}\colon X^{n+1}\to X$ be the $i$-th projection. Let  $J=\{j_1,\ldots,j_m\}$ where $1\le j_1<\ldots < j_{t}\le(n+1)$, in particular ${t}=|J|$. We let $\pi_{J}\colon X^{n+1}\to X^{t}$ be the map such that the composition of the $i$-th projection $X^{t}\to X$ with $\pi_J$ is equal to $\rho_{j_i}$. The two maps $\pi_I\colon X^{n+1}\to X^{{t}}$ and $\pi_{I^c}\colon X^{n+1}\to  X^{n+1-{t}}$ define an isomorphism $\Lambda_I\colon X^{n+1}\overset{\sim}{\lra} X^{t}\times X^{n+1-{t}}$. We have
\begin{equation}
(f^{n+1})^{*}\Delta_I(Y;b)=\Lambda_I^{*}((f^{t})^{*}\Delta^{t}(Y)\times (f^{n+1-{t}})^{*}(\{\underbrace{(b,\ldots,b)}_{n+1-{t}}\})).
\end{equation}
(Here $\times$ denotes the exterior product of cycles, see 1.10 of~\cite{fulton}.) An obstruction bundle computation gives that 
\begin{equation}
(f^{n+1-{t}})^{*}(\{\underbrace{(b,\ldots,b)}_{n+1-{t}}\})=\beta_{n+1-{t},*}\left(\prod_{1\le j\le (n+1-{t})} c_{e-1}(Q_j(n+1-{t})\right)
\end{equation}
The corollary follows from the above equations and~\Ref{prp}{eccesso}. 
\end{proof}
Let $I\subset\{1,\ldots,(n+1)\}$ be non-empty and let $t:=| I |$. We let $\Omega_I\in \CH_{\dim X}(E^{n+1})$ be given by 
\begin{equation}\label{eccomega}
\scriptstyle
\Omega_I:=
\begin{cases}
\scriptstyle 0 & \scriptstyle \text{if $| I |=1$,} \\
\scriptstyle (g^{n+1})^{*}\Delta_I(V;b)\cdot c_{d({t})}\left(\bigoplus\limits_{j\in I} Q_j\right)\cdot \prod_{j\in I^c} c_{e-1}(Q_j) & \scriptstyle\text{if $| I |> 1$.}
\end{cases}
\end{equation}
By~\Ref{crl}{eccesso} we have $(f^{n+1})^{*}\Delta_I(Y;b)= \Delta_I(X;a)+\beta_{n+1,*}(\Omega_I)$ and hence
\begin{equation}\label{solgam}
(f^{n+1})^{*}(\Gamma^{n+1}(Y;b))=\Gamma^{n+1}(X;a)+\beta_{n+1,*}\left(\sum_{1\le| I |\le(n+1)}(-1)^{n+1-| I |}\Omega_I\right).
\end{equation}
\subsection{The proof.}
\setcounter{equation}{0}
By~\eqref{solgam} it suffices to prove that the following equality holds in $\CH_{\dim X}(E^{n+1})_{\QQ}$:
\begin{equation}\label{altome}
\sum_{1\le| I |\le(n+1)}(-1)^{| I |}\Omega_I=0.
\end{equation}
Let $I\subset\{1,\ldots,(n+1)\}$ be of cardinality strictly greater than $(n-e)$: \Ref{crl}{delsup}  allows us to express the class of $\Delta_I(V;b)$ as a linear combination of the $\Delta_J(V;b)$'s with $J\subset I$  of cardinality at most $(n-e)$. Moreover Whitney's formula allows us to write the Chern class appearing in the definition of $\Omega_I$ as a sum of products of Chern classes of the $Q_j$'s.  
It follows that for each $I\subset\{1,\ldots,(n+1)\}$ we may express the class of $\Omega_I$ as a 
linear combination of the classes
\begin{equation}
(g^{n+1})^{*}\Delta_J(V;b)\cdot \prod_{s=1}^{n+1}c_{k_s}(Q_s),\quad 1\le |J|\le(n-e),\quad k_1+\ldots+k_{n+1}=d(n+1)=n(e-1)-1.
\end{equation}
\begin{dfn}
$\cP_n(e)$ is the set of $(n+1)$-tuples $k_1,\ldots,k_{n+1}$ of natural numbers $0\le  k_s\le (e-1)$ whose sum equals $d(n+1)$.
\end{dfn}
Summing over all $I\subset\{1,\ldots,(n+1)\}$  of a given cardinality  $t$ we get the following.
\begin{clm}
Let $1\le t\le(n+1)$. There exists an integer $c_{J,K}(t)$ for each  couple $(J,K)$ with $\es\not=J\subset\{1,\ldots,(n+1)\}$  of cardinality at most $(n-e)$ and $K\in\cP_n(e)$   such  that
\begin{equation}
\sum_{| I |=t}\Omega_I=\sum_{\substack{1\le | J |\le (n-e) \\ K\in\cP_n(e)}}c_{J,K}(t)(g^{n+1})^{*}\Delta_J(V;b)\cdot \prod_{s=1}^{n+1}c_{k_s}(Q_s).
\end{equation}
\end{clm}
It will be convenient to set $c_{J,K}(0)=0$.  We will prove that  
\begin{equation}\label{incredibile}
\sum_{t=0}^{n+1}(-1)^{t} c_{J,K}(t)=0.
\end{equation}
That will prove Equation~\eqref{altome} and hence also~\Ref{prp}{blowdel}. Applying~\Ref{crl}{delsup} to $(V,b)$ we get the following result.
\begin{clm}\label{clm:ritrito}
Let $I\subset\{1,\ldots,n+1\}$ be of cardinality $t\ge (n+1-e)$. Then
\begin{equation}
\Delta^{n+1}_I(V;b)\equiv \sum_{\substack{J\subset I \\ 1\le |J|\le(n-e)}}(-1)^{n-e-|J|}{t-|J| -1\choose t-n-1+e}\Delta^{n+1}_J(Y;b).
\end{equation}
\end{clm}
Given $K\in\cP_n(e)$ we let
\begin{equation}
T(K):=\{1\le i\le(n+1) \mid k_i=(e-1)\}.
\end{equation}
A simple computation gives that
\begin{equation}\label{tikapineq}
(n+1-e)\le | T(K)|. 
\end{equation}
\begin{prp}
Let $\es\not=J\subset\{1,\ldots,(n+1)\}$ be of cardinality at most $(n-e)$, let $K\in\cP_n(e)$ and $0\le t\le(n+1)$. Then
\begin{equation}\label{cigeikap}
c_{J,K}(t)=(-1)^{n-|J|-e}{t-|J| -1\choose n-|J|-e}{|T(K)\cap J^c|\choose n+1-t}
\end{equation}
\end{prp}
\begin{proof}
Suppose first that $0\le t\le(n-e)$. Then $c_{J,K}(t)=0$ unless $|J|=t$ and $J^c\subset T(K)$: if the latter holds then $c_{J,K}(t)=1$. Assume that the right-hand side of~\eqref{cigeikap} 
is non-zero: then the first binomal  coefficient is non-zero and hence $t\le|J|$. Of course also the second binomal  coefficient is non-zero: it follows that
\begin{equation}
(n+1-t)\le | T(K)\cap J^c| \le |J^c|=n+1-|J|.
\end{equation}
Since $t\le|J|$ it follows that $|J|=t$ and hence $| T(K)\cap J^c| = |J^c|$ i.e.~$J^c\subset T(K)$: a straightforward computation gives that under these assumptions the right-hand side of~\eqref{cigeikap} equals $1$. 
It remains to prove that~\eqref{cigeikap} holds for $(n+1-e)\le t\le(n+1)$. Looking at~\eqref{eccomega} and~\Ref{clm}{ritrito} we get that
\begin{equation}\label{lucca}
c_{J,K}(t)=(-1)^{n-e-|J|}{t-|J| -1\choose t-n-1+e}|\{I\subset\{1,\ldots,(n+1)\} \mid  I^c\subset(T(K)\cap J), \quad |I|=t\}|.
\end{equation}
Since the right-hand side of~\eqref{lucca} is equal to the right-hand side of~\eqref{cigeikap} this finishes the proof.
\end{proof}
Let
\begin{equation}
p(x):={n-|J| -x\choose n-|J|-e}.
\end{equation}
Then $\deg p<|T(K)\cap J^c|$ because  $\deg p=(n-|J|-e)$ and because~\eqref{tikapineq} gives that 
\begin{equation}
|T(K)\cap J^c|\ge (n+1-e)+(n+1-|J|)-(n+1)=n-|J|-e+1.
\end{equation}
Thus~\eqref{combcomb} and~\eqref{cigeikap}    give that
\begin{equation}
\scriptstyle
0=\sum_{s=0}^{n+1}(-1)^s p(s) {|T(K)\cap J^c|\choose s} =(-1)^{n+1}\sum_{t=0}^{n+1}(-1)^t {t-|J| -1\choose n-|J|-e}{|T(K)\cap J^c|\choose n+1-t}=
(-1)^{1-e-|J|}\sum_{t=0}^{n+1}(-1)^t c_{J,K}(t).
\end{equation}
This finishes the prooof of~\Ref{prp}{blowdel}.
\qed
\subsection{Application to Hilbert schemes of $K3$'s}
Let $S$ be a complex $K3$ surface. By Beauville and Voisin~\cite{beauvoisin} there exists $c\in S$ such that $\Gamma^3(S;c)\equiv 0$. We let $S^{[n]}$ be the Hilbert scheme parametrizing length-$n$ subschemes of $S$;  Beauville~\cite{beau} proved that $S^{[n]}$ is a hyperk\"ahler variety. 
\begin{prp}\label{prp:diaghilbk3}
Keep notation as above and assume that $n=2,3$. Let $a_n\in S^{[n]}$  represent a scheme supported at $c$. Then $\Gamma^{2n+1}(S^{[n]};a_n)\equiv 0$. 
\end{prp}
\begin{proof}
First assume that $n=2$. Let $\pi_1\colon X\to S\times S$ be the blow-up of the diagonal $\Delta$ and $\rho_2\colon X\to S^{(2)}$ the composition of $\pi_1$ and the quotient map $S\times S\to S^{(2)}$. There is a  degree-$2$ map $\phi_2\colon X\to S^{[2]}$ fitting into a commutative diagram
\begin{equation}\label{equivoci}
\xymatrix{ X \ar^{\phi_2}[r]  \ar^{\rho_2}[dr]  &   S^{[2]} \ar^{\gamma_2}[d] \\
  &  S^{(2)}  } 
\end{equation}
where $\gamma_2([Z])=\sum_{p\in S}\ell(\cO_Z,p)$ is the Hilbert-Chow morphism. Let $x\in X$ such that $\phi_2(x)=a_2$; by~\Ref{subsec}{gradofin} it suffices to prove that $\Gamma^5(X;x)\equiv 0$. By commutativity of~\eqref{equivoci} we have $\pi_1(x)=(c,c)$. Now $\Gamma^5(S\times S;(c,c))\equiv 0$ by~\Ref{prp}{protokunn}, and since $\cod(\Delta,S\times S)=2$ it follows from~\Ref{prp}{blowdel} that $\Gamma^5(X;x)\equiv 0$. Next assume that $n=3$. Let  $\pi_2\colon Y\to S^{[2]}\times S$ be the blow-up with center  the tautological subscheme  $\cZ_2\subset S^{[2]}\times S$ and  $\rho_3\colon Y\to S^{(3)}$ the composition of $\pi_2$ and the natural map $S^{[2]}\times S\to S^{(3)}$. There is a  degree-$3$   map  $\phi_3\colon Y\to S^{[3]}$ fitting into a commutative diagram
\begin{equation}\label{commedia}
\xymatrix{ Y \ar^{\phi_3}[r]  \ar^{\rho_3}[dr]  &   S^{[3]} \ar^{\gamma_3}[d] \\
  &  S^{(3)}  } 
\end{equation}
where $\gamma_3$ is the Hilbert-Chow morphism. (See for example Proposition~2.2 of~\cite{ellstrom}.) On the other hand let $p_1\colon S\times S\to S$ be projection to the first factor;   the map 
\begin{equation}\label{nicolapiazza}
(\phi_2,p_1\circ \pi_1))\colon X\to S^{[2]}\times S 
\end{equation}
is an isomorphism onto $\cZ_2$. Let $y\in Y$ be such that $\phi_3(y)=a_3$; by~\Ref{subsec}{gradofin} it suffices to prove that $\Gamma^7(Y;y)\equiv 0$. Notice that $\pi_2(y)=(a_2,c)$ where $a_2\in S^{[2]}$ is supported at $c$. By the case $n=2$ (that we just proved) and~\Ref{prp}{protokunn} we have $\Gamma^7(S^{[2]}\times S;(a_2,c))\equiv 0$.  Let $x\in X$ such that $\phi_2(x)=a_2$. In the proof for the case $n=2$ we showed that $\Gamma^5(X;x)\equiv 0$; since~\eqref{nicolapiazza} is an isomorphism it follows that  $\Gamma^5(\cZ_2;(a_2,c))\equiv 0$. Since  $\Gamma^7(S^{[2]}\times S;(a_2,c))\equiv 0$ and $\cZ_2$ is smooth of codimension $2$, we get  
$\Gamma^7(Y;y)\equiv 0$ by~\Ref{prp}{blowdel}.
\end{proof}
 Let $\cZ_n\subset S^{[n]}\times S$ be the tautological subscheme. The blow-up of $S^{[n]}\times S$ with center $\cZ_n$ has a natural regular  map of finite (non-zero) degree  to $S^{[n+1]}$ and in turn $\cZ_n$ may be described starting from  the tautological subscheme  $\cZ_{n-1}\subset S^{[n-1]}\times S$. Thus one may hope to prove by induction on $n$ that  $\Gamma^{2n+1}(S^{[n]};a)\equiv 0$ for any $n$: the problem  is that  starting with $\cZ_{3}$ the tautological subscheme is singular. 
\section{Double covers}\label{sec:rivdop}
In the present section we will assume that $X$ is a   projective variety  over a field $\KK$ and that $\iota\in\Aut(X)$ is a (non-trivial) involution. We let $Y:=X/\la\iota\ra$  and 
$f\colon X\to Y$ be the quotient map. 
We  assume that there exists $a\in X(\KK)$ which is fixed by $\iota$ and we let $b:=f(a)$.  
\begin{cnj}\label{cnj:raddoppia}
Keep hypotheses and notation as above and  
suppose that  $\Gamma^{m}(Y;b)\equiv 0$.  Then  $\Gamma^{2m-1}(X;a)\equiv 0$. 
\end{cnj}
The above conjecture was proved for $m=2$ by Gross and Schoen, see Prop.~4.8 of~\cite{groscho}. We will propose a proof of~\Ref{cnj}{raddoppia} and we will show that the proof works for $m=2,3$. Of course the proof for $m=2$ is that of  Gross and Schoen (with the  symmetric cube of the curve replaced by the  cartesian cube).
\subsection{A modest proposal}\label{subsec:insintesi}
\setcounter{equation}{0}
 There is a well-defined pull-back homomorphisms 
\begin{equation}
(f^q)^{*}\colon Z_{*}(Y^q)_{\QQ}\to Z_{*}(X^q)_{\QQ}
\end{equation}
compatible with rational equivalence (see Ex.~1.7.6 of~\cite{fulton}): thus we have an   induced   homomorphism $(f^q)^{*}\colon\CH_{*}(Y^q)_{\QQ}\to \CH_{*}(X^q)_{\QQ}$. Let $n:=\dim X$  and $\Xi_m\in Z_{n}(X^m)_{\QQ}$ the cycle defined by 
\begin{equation}\label{dramper}
\Xi_m:=(f^m)^{*}\Gamma^m(Y;b).
\end{equation}
We will show that  $\Xi_m$ is a linear combination of cycles  of the type 
\begin{equation}\label{tipi}
\{(x,\ldots,\iota(x),\ldots x,\ldots,x,a,\ldots\iota(x),\ldots,a,\ldots) \mid x\in X\}.
\end{equation}
Notice that the $\Delta_I(X;a)$'s are of this type. 
Consider the inclusions of $X^m$ in $X^{2m-1}$ which map $(x_1,\ldots,x_{m})$ to $(x_1,\ldots,x_{m},\nu(1),\ldots,\nu(m-1))$ where $\nu\colon\{1,\ldots,(m-1)\}\to\{a,x_1,\ldots,x_{m},\iota(x_1),\ldots,\iota(x_{m})\}$ is an arbitrary list. 
Let $\Phi_{\nu}(\Xi_m)$ be the symmetrized image of  $\Xi_m$ in $Z_{n}(X^{2m-1})$ for the inclusion determined by $\nu$: it is a  linear combination of cycles~\eqref{tipi}.
 By hypothesis $\Xi_m\equiv 0$ and hence any linear combination of the  cycles  $\Phi_{\nu}(\Xi_m)$ is rationally equivalent to $0$. One gets the proof if a suitable linear combination of  the $\Phi_{\nu}(\Xi_m)$'s is a linear combination of  the $\Delta_I(X;a)$'s with the appropriate coefficients (so that it is equal to a non-zero multiple of $\Gamma^{2m-1}(X;a)$). We will carry out the proof for $m=2,3$. 
\subsection{Preliminaries}\label{subsec:giocomega}
\setcounter{equation}{0}
Since the involution of $X$ is non-trivial the dimension of $X$ is strictly positive i.e.~$n>0$. Let  $\mu\colon\{1,\ldots,q\}\to\{a,x,\iota(x)\}$. If $\mu$ is \emph{not} the sequence $\mu(1)=\ldots=\mu(q)=a$ we let
\begin{equation}
\Omega(\mu(1),\ldots,\mu(q)):=\{(x_1,\ldots,x_{q})\in X^{q} \mid x_i=\mu(i),\quad x\in X\},
\end{equation}
 and we let $\Omega(a,\ldots,a):=0$. 
Thus  $\Omega(\mu(1),\ldots,\mu(d))$ is  an $n$-cycle on $X^{d}$. 
For example $\Omega(x,\ldots,x)\in X^{q}$ is the small diagonal.   Let $\cS_{q}$ be the symmetric group on $\{ 1,\ldots,q\}$: of course it acts on $X^{q}$. For $r+s+t=q$ let
\begin{equation}
\ov{\Omega}(r,s,t):=\sum_{\sigma\in\cS_{q}}\sigma(\Omega(\underbrace{a,\ldots,a}_{r},\underbrace{x,\ldots,x}_{s},
\underbrace{\iota(x),\ldots,\iota(x)}_{t})).
\end{equation}
Thus $\ov{\Omega}(r,s,t)$ is an $n$-cycle on $X^q$  invariant under the action of $\cS_{q}$.  Notice that 
\begin{equation}\label{esseti}
\ov{\Omega}(r,s,t)=\ov{\Omega}(r,t,s). 
\end{equation}
With this notation
\begin{equation}\label{altdiag}
\Gamma^q(X;a)=\sum_{\substack{ 0\le r,s \\ r+s=q}} \frac{(-1)^r}{r! s!}\ov{\Omega}(r,s,0).
\end{equation}
Let $\Xi_m$ be the cycle on $X^m$ given by~\eqref{dramper}. A straightforward computation gives that
\begin{equation}\label{tirodiag}
2\Xi_m=\sum_{\substack{ 0\le r,s,t \\ r+s+t=m}}\frac{(-2)^r}{r! s! t!}\ov{\Omega}(r,s,t).
\end{equation}
(Equality~\eqref{esseti} is the reason for the factor of $2$ in front of $\Xi_m$.)  
For  
\begin{equation*}
\nu\colon\{1,\ldots,(m-1)\}\to\{a,x_1,\ldots,x_{m},\iota(x_1),\ldots,\iota(x_{m})\}
\end{equation*}
we let
\begin{equation}
\begin{matrix}
X^{m}& \overset{j_{\nu}}\lra &  X^{2m-1} \\
(x_1,\ldots,x_{m}) & \mapsto & (x_1,\ldots,x_{m},\nu(1),\ldots,\nu(m-1))
\end{matrix}
\end{equation}
and $\Phi_{\nu}\colon Z_n(X^{m})\to Z_n(X^{2m-1})$ be the homomorphism 
\begin{equation}
\Phi_{\nu}(\gamma):=\sum_{\sigma\in\cS_{2m-1}}\sigma_{*}(j_{\nu,*}(\gamma)).
\end{equation}
Notice that $\Phi_{\nu}$ does not change if we reorder the sequence $\nu$.
\subsection{The case $m=2$}\label{subsec:emmedue}
\setcounter{equation}{0}
 A straightforward computation (recall~\eqref{esseti}) gives  that
\begin{eqnarray}
\Phi_a(\Xi_2) & = & \ov{\Omega}(1,2,0)-4\ov{\Omega}(2,1,0)+\ov{\Omega}(1,1,1), \\
\Phi_{x_1}(\Xi_2) & = & \ov{\Omega}(0,3,0)-2\ov{\Omega}(1,2,0)-2\ov{\Omega}(2,1,0)+\ov{\Omega}(0,2,1), \\
\Phi_{\iota(x_1)}(\Xi_2) & = & -2\ov{\Omega}(2,1,0)-2\ov{\Omega}(1,1,1)+2\ov{\Omega}(0,2,1). 
\end{eqnarray}
Thus
\begin{equation}
0\equiv -2\Phi_a(\Xi_2)+2\Phi_{x_1}(\Xi_2)-\Phi_{\iota(x_1)}(\Xi_2)=2\ov{\Omega}(0,3,0)-6\ov{\Omega}(1,2,0)+6\ov{\Omega}(2,1,0)=12\Gamma^3(X;a).
\end{equation}
\subsection{The case $m=3$}\label{subsec:emmetre}
\setcounter{equation}{0}
For every  $\nu\colon\{1,2\}\to \{a,x_1,x_2,x_3,\iota(x_1),\iota(x_2),\iota(x_3)\}$ the cycle $\Phi_\nu(\Xi_3)$ is equal to the linear combination of the classes listed in the first column of Table~\eqref{coordinate}  with  coefficients  the numbers in the corresponding column   of Table~\eqref{coordinate}. For such a $\nu$ let $i(\nu)$ be its position in the first row of Table~\eqref{coordinate}: thus $i((a,a))=1$,..., $i((\iota(x_1),\iota(x_2))=9$. Table~\eqref{coordinate} allows us to rewrite 
\begin{equation}\label{califano}
\sum_{\nu} \lambda_{i(\nu)}\Phi_{\nu}(\Xi_3)
\end{equation}
 as an integral linear combination  of the classes listed in the first column of Table~\eqref{coordinate}, with coefficients $F_1,\ldots,F_9$ which are linear functions of $\lambda_1,\ldots,\lambda_9$. Let's impose that $0=F_1=\ldots=F_6$: solving the corresponding  linear system we 
get that
 \begin{eqnarray}
\lambda_1 & = & \frac{1}{3}(-8\lambda_6-2\lambda_7-8\lambda_8-8\lambda_9),\\
\lambda_2 & = & \frac{1}{3}(14\lambda_6+8\lambda_7+14\lambda_8+20\lambda_9),\\
\lambda_3 & = & \frac{1}{3}(-6\lambda_6-6\lambda_7-6\lambda_8-12\lambda_9),\\
\lambda_4 & = & \frac{1}{3}(\lambda_6-2\lambda_7+\lambda_8+4\lambda_9),\\
\lambda_5 & = & \frac{1}{3}(-5\lambda_6-2\lambda_7-5\lambda_8-8\lambda_9).
\end{eqnarray}
For such a choice of coefficients $\lambda_1,\ldots,\lambda_9$ we have that
\begin{equation}\label{pasqua}
%
0\equiv \sum\limits_{\nu} \lambda_{i(\nu)}\Phi_{\nu}(\Xi)=
-\frac{4}{3}(\lambda_6+\lambda_7+\lambda_8+\lambda_9)(\ov{\Omega}(0,5,0)-5\ov{\Omega}(1,4,0)+10\ov{\Omega}(2,3,0)
-10\ov{\Omega}(3,2,0)+5\ov{\Omega}(4,1,0)).
\end{equation}
Choosing integers $\lambda_6,\ldots,\lambda_9$  such that $(\lambda_6+\lambda_7+\lambda_8+\lambda_9)=-3$ we get that
\begin{equation}\label{eccoci}
0\equiv \sum\limits_{\nu} \lambda_{i(\nu)}\Phi_{\nu}(\Xi)=4\cdot 5!\Gamma^5(X;a).
\end{equation}
This concludes the proof of~\Ref{cnj}{raddoppia} for $m=3$.
\begin{table}[tbp]\tiny
\caption{Coordinates of $\Phi_{\nu}(\Xi)$ for  $\nu=(a,a),\ldots,(\iota(x_1),\iota(x_2))$.}\label{coordinate}
\vskip 1mm
\centering
\renewcommand{\arraystretch}{1.60}
\begin{tabular}{rrrrrrrrrr}
\toprule
& $ (a,a)$  & $ (a,x_1)$   &  $ (a,\iota(x_1))$    &  $(x_1,x_1)$  &  $(x_1,x_2)$   &  
$(x_1,\iota(x_1))$  &  $(x_1,\iota(x_2))$    &  $ (\iota(x_1),\iota(x_1))$  & $ (\iota(x_1),\iota(x_2))$     \\
\midrule
$\ov{\Omega}(3,1,1)$ & -6 & -2 & 2 & -2 & 0 & -2 & 4 & -2 & 8    \\
 \midrule
$\ov{\Omega}(2,2,1)$   & 3  & -4 & -8 & 0 & -4 & 4 & -6 & 4 & -8      \\
  \midrule
$\ov{\Omega}(1,3,1)$    & 0 & 2 & 2  & -4 & 0 & -4 & -4 & -4 & 0     \\
 \midrule
$\ov{\Omega}(1,2,2)$    & 0 & 1 & 2 & 0 &-2  &-4  & 0 & -4 & -4       \\
\midrule
$\ov{\Omega}(0,4,1)$    & 0 & 0 & 0 & 2 & 1 &1  & 2 & 1 & 0     \\
\midrule
$\ov{\Omega}(0,3,2)$     & 0 & 0 & 0 & 1 & 2 &3  & 2 & 3 & 4         \\
\toprule
$\ov{\Omega}(0,5,0)$     & 0 & 0 & 0 & 1 & 1 & 0 & 0 & 0 & 0        \\
\midrule
$\ov{\Omega}(1,4,0)$    & 0 & 1 & 0 & -4 & -2 & 0 & 0 & 0 & 0      \\
\midrule
$\ov{\Omega}(2,3,0)$    & 1 & -4 & 0 & 4 & -4 & 0 & -2 & 0 & 0     \\
\midrule
$\ov{\Omega}(3,2,0)$    & -6 & 2 & -2 & -2 & 8 & -2 & 4 & -2 & 0       \\
\midrule
$\ov{\Omega}(4,1,0)$      & 12 & 8 & 8  & 8 & 4 & 8 & 4 & 8 & 4         \\
\bottomrule 
\end{tabular}
\end{table} 
\end{document}